\documentclass[10pt,leqno,a4paper,oneside]{amsart}
\usepackage{amsmath,soul,enumitem,setspace,geometry}
\usepackage{amsfonts}
\usepackage{amssymb}
\usepackage{amsthm}
\usepackage{lmodern}
\usepackage[T1]{fontenc}
\usepackage{mathrsfs}
\usepackage{stmaryrd}
\usepackage{bussproofs}
\usepackage{phonetic}
\usepackage{upgreek}
\usepackage[applemac]{inputenc}
\usepackage{cite}
    \usepackage{rotating}
        \usepackage{verbatim}
\usepackage[textwidth=2.5cm]{todonotes}

  \let\origmaketitle\maketitle
\def\maketitle{
  \begingroup
  \def\uppercasenonmath##1{} 
  \let\MakeUppercase\relax 
  \origmaketitle
  \endgroup
}

\newtheorem{definition}{Definition}
\newtheorem{theorem}{Theorem}
\newtheorem{lemma}{Lemma}
\newtheorem{corollary}{Corollary}

\newtheorem{proposition}{Proposition}

\theoremstyle{definition}
\newtheorem{remark}{Remark}

\newcommand{\T}{{\tt Tr}}

\newcommand{\lra}{\leftrightarrow}
\newcommand{\ovl}{\overline}
\long\def\symbolfootnote[#1]#2{\begingroup%
\def\thefootnote{\fnsymbol{footnote}}\footnote[#1]{#2}\endgroup}

\newcommand{\ra}{\rightarrow}

\newcommand{\Ra}{\Rightarrow}
\newcommand{\nat}{\mbb{N}}

\newcommand{\mbb}{\mathbb}

\newcommand{\vphi}{\varphi}
\newcommand{\sth}{\;|\;}

\newcommand{\urc}{\urcorner}
\newcommand{\ulc}{\ulcorner}

\newcommand{\mc}{\mathcal}

\newcommand{\mf}{\mathfrak}

\newcommand{\corn}[1]{\ulc#1\urc}

\newcommand{\lt}{\mc L_{\T}}

\newcommand*\subdot[1]{\oalign{$#1$\cr\hfil.\hfil}}%
\providecommand\dvee%
		{\mathbin{\subdot \vee}}
\providecommand\dneg%
		{\mathbin{\subdot \neg}}
\providecommand\dwedge%
		{\mathbin{\subdot \wedge}}
\providecommand\drightarrow%
		{\mathbin{\subdot \rightarrow}}
\providecommand\dnot%
		{\mathord{\subdot{\lnot}}}
\providecommand\dbot%
		{\mathord{\subdot{\bot}}}
\providecommand\dforall%
		{\mathord{\subdot\forall\!\!}}
\providecommand\dexists%
		{\mathord{\subdot\exists}}
\newcommand{\sentkf}{{\tt Sent}_{\lt}}

\newcommand{\ct}{{\tt Cterm}_{\lt}}

\newcommand{\lnat}{\mc{L}_\nat}

\EnableBpAbbreviations

  \newcommand{\kf}{{\bf KF}}
  \newcommand{\pkf}{{\bf PKF}}

\newcommand{\kfl}{{{\bf KFL}}}
\newcommand{\kfls}{{\bf KFL^*}}
\theoremstyle{definition}

\newcommand{\puu}[1]{\,\text{\d{$#1$}}\,}
	\newenvironment{enumeratei}{\begin{enumerate}%
[label={\textup{(\roman*)}}]\setlength{\itemsep}{.8ex}}%
{\end{enumerate}}
	\newenvironment{itemize-}{\begin{itemize}[label={-}]\setlength{\itemsep}{.8ex}}{\end{itemize}}

\title[Nonclassical truth with classical strength.]{Nonclassical truth with classical strength. \\{\normalfont A proof-theoretic analysis of compositional truth over {\sc hype}.}}
\author{Martin Fischer and Carlo Nicolai and Pablo Dopico Fernandez}

\begin{document}


\maketitle
\begin{abstract}
Questions concerning the proof-theoretic strength of classical versus nonclassical theories of truth have received some attention recently. A particularly convenient case study concerns classical and nonclassical axiomatizations of fixed-point semantics. It is known that nonclassical axiomatizations in four- or three-valued logics are substantially weaker than their classical counterparts. In this paper we consider the addition of a suitable conditional to First-Degree Entailment -- a logic recently studied by Hannes Leitgeb under the label {\bf HYPE}. We show in particular that, by formulating the theory {\bf PKF} over {\bf HYPE}, one obtains a theory that is sound with respect to fixed-point models, while being proof-theoretically on a par with its classical counterpart {\bf KF}. Moreover, we establish that also its schematic extension -- in the sense of Feferman -- is as strong as the schematic extension of {\bf KF}, thus matching the strength of predicative analysis.
\end{abstract}

\section{Introduction}

The question whether there are nonclassical formal systems of primitive truth that can achieve significant proof-theoretic strength has received much attention in the recent literature. Solomon Feferman \cite{fef84} famously claimed that `nothing like sustained ordinary reasoning can be carried on' in the standard nonclassical systems that support strong forms of inter-substitutivity of $A$ and `$A$ is true'. One way of understanding this claim is by measuring how much mathematics can be encoded in such systems. Since the strength of mathematical systems (whether classical or nonclassical) is traditionally measured in terms of the ordinals that can be well-ordered by them, the ordinal analysis of nonclassical systems of truth becomes relevant.

We are mainly interested in the proof-theoretic analysis of nonclassical systems inspired by fixed-point semantics \cite{kri75}. Since fixed-point semantics has nice axiomatizations, both classical and nonclassical, it represents a particularly convenient arena to measure the impact of weakening the logic on proof-theoretic strength. The axiomatization of fixed-point semantics in classical logic -- a.k.a.~$\kf$ -- is known to have the proof-theoretic ordinal $\vphi_{\varepsilon_0}0$ \cite{fef91,can89}.\footnote{Or  $\Gamma_0$, depending on whether one focuses on a version of the theory with or without suitable open-ended substitution rule schemata.} Halbach and Horsten have proposed in \cite{haho06} a nonclassical axiomatization, known as $\pkf$, and showed that it has proof-theoretic ordinal $\vphi_\omega 0$. There have been some attempts to overcome this mismatch in strength on the nonclassical side. \cite{nic17} showed that even without expanding the logical resources of the  theory, $\pkf$ can be extended with suitable instances of transfinite induction to recover all classical true theorems of $\kf$.  \cite{fihoni17} showed that a simple theory featuring nonclassical initial sequents of the form $A\Ra \T\corn{A}$ and $\T\corn{A}\Ra A$ can be closed under special reflection principles to recover the arithmetical strength of $\pkf$ and $\kf$. More recently, \cite{fie20} showed that, by enlarging the primitive concepts of $\pkf$ with a predicate for `classicality', one can achieve the proof-theoretic strength of $\kf$ in both the schematic and non-schematic versions. 

In the paper we explore a different option, which in a sense completes the picture above. We enlarge the standard four-valued logic of $\pkf$ with a new conditional, which is based on the logic ${\bf HYPE}$ recently proposed by \cite{lei19}. The conditional has several features that resemble an intuitionistic conditional, but its weaker interaction with the ${\bf FDE}$ -negation makes it possible to sustain the intersubstitutivity of $A$ and `$A$ is true' for sentences not containing the conditional. This extended theory, that we call $\kfl$, is shown to be proof-theoretically equivalent to $\kf$. Its extension with a schematic substitution rule, called $\kfls$, is shown to be proof-theoretically equivalent to the schematic extension of $\kf$ -- called ${\tt Ref}^*({\tt PA}(P))$ in \cite{fef91}.

In particular, we show  that the conditional of the logic ${\bf HYPE}$ enables one to mimic, when carefully handled, the standard lower bound proofs by Gentzen and Feferman-Sch\"utte for transfinite induction in classical arithmetic (Theorem \ref{thm:genhyp}) and predicative analysis (Proposition \ref{prop:kflstar}), respectively. This enables us to define, in our theories $\kfl$ and $\kfls$, ramified truth predicates indexed by ordinals smaller than $\varepsilon_0$ (Corollary \ref{cor:lowbound}) and $\Gamma_0$ (Corollary \ref{cor:lbkfls}). Moreover, the proof-theoretic analysis of $\kfl$ and $\kfls$ is completed by showing that their truth predicates can be suitably interpreted in their classical counterparts $\kf$ and ${\tt Ref}^*({\tt PA}(P))$ without altering the arithmetical vocabulary (Propositions \ref{prop:ubkfl} and \ref{prop:upschm}).

\section{HYPE}

In this section we will present the logical basis of our systems of truth. We will work with a sequent calculus variant of the logic ${\bf HYPE}$ introduced by Leitgeb in \cite{lei19} by means of a Hilbert style calculus. Essentially, the calculus is obtained by extending First-Degree Entailment with an intuitionistic conditional and with rules for it in a multi-conclusion style. 

\subsection{${\bf  G1h_{cd}}$}

We present a multi-conclusion system based on a multi-conclusion calculus for intuitionistic logic:\footnote{This system goes back to Maehara's version used in Takeuti \cite{tak87} p.52f and Dragalin's system used in Negri and Plato \cite{nepl01} p.108f.} we call it ${\bf  G1h_{cd}}$ for Gentzen system for the logic ${\bf HYPE}$ with constant domains. Sequents are understood as multisets. We work with a language whose logical symbols are $\neg$, $\vee$ , $\ra$, $\forall$, $\bot$. For $\Gamma=\gamma_1,\ldots,\gamma_n$ a multiset, $\neg \Gamma$ is the multiset $\neg\gamma_1,\ldots,\neg\gamma_n$. The logical constants $\wedge, \exists,\lra$ can be defined as usual and $\top$ is defined as $\neg \bot$. Moreover, we can define  `intuitionistic' negation $\sim A$ as $A \rightarrow \bot$, the material conditional $A\supset B$ as $\neg A\vee B$, and  material equivalence $A\equiv B$ as $(A\supset B)\land (B\supset A)$. For $A$ a formula, we write ${\tt FV}(A)$ for the set of its free variables, and ${\tt FV}(\Gamma)$ for the set of free variables in all formulas in $\Gamma$. 

The system ${\bf  G1h_{cd}}$ consists of the following initial sequents and rules:
\[ ({\tt ID}_p) \hspace{1cm} A \Ra A  \hspace{3cm} ({\sf L \bot}) \hspace{1cm} \bot \Ra \]
\begin{center}

\AxiomC{$ \Gamma \Ra  \Delta, A $}
		\AxiomC{$A, \Gamma \Ra \Delta $}
		\LeftLabel{(\texttt{Cut})}
		\BinaryInfC{$ \Gamma \Ra  \Delta $}
		\DisplayProof
		
\end{center}

\begin{center}

\begin{tabular}{c c}
	
		&\\
		
		\AxiomC{$\Gamma \Ra \Delta$}
		\LeftLabel{$({\tt L W})$}
		\UnaryInfC{$A, \Gamma \Ra \Delta$}
		\DP
		
		&
		
		\AxiomC{$\Gamma \Ra \Delta$}
		\LeftLabel{$({\tt R W})$}
		\UnaryInfC{$\Gamma \Ra \Delta, A$}
		\DisplayProof
		
		\\
		& \\
		
		\AxiomC{$A,A, \Gamma \Ra \Delta$}
		\LeftLabel{$({\tt L C})$}
		\UnaryInfC{$A, \Gamma \Ra \Delta$}
		\DisplayProof
		
		&
		
		\AxiomC{$\Gamma \Ra \Delta, A,A$}
		\LeftLabel{$({\tt R C})$}
		\UnaryInfC{$\Gamma \Ra \Delta, A$}
		\DisplayProof
		
		\\
		& \\
		
		\AxiomC{$ A, \Gamma \Ra  \Delta $}
		\AxiomC{$ B, \Gamma \Ra  \Delta $}
		\LeftLabel{(\texttt{L$\vee$})}
		\BinaryInfC{$ A \vee B, \Gamma \Ra \Delta$}
		\DisplayProof
		
		&
		
		\AxiomC{$ \Gamma \Ra  A, B,  \Delta$}  \LeftLabel{(\texttt{R$\vee$})}
		\UnaryInfC{$ \Gamma \Ra  A \vee B,  \Delta$}
		\DisplayProof
		
		\\
			&\\
		
		\AxiomC{$ \Gamma \Ra \Delta, A $}
		\AxiomC{$ B, \Gamma \Ra \Delta $}
		\LeftLabel{(\texttt{L$\rightarrow$})}
		\BinaryInfC{$ A \rightarrow B, \Gamma \Ra \Delta$}
		\DisplayProof
		
		&
		
		\AxiomC{$ \Gamma , A \Ra  B $} \LeftLabel{(\texttt{R$\rightarrow$})}
		\UnaryInfC{$ \Gamma \Ra  A \rightarrow B, \Delta$}
		\DisplayProof

\\ & \\
				
		\AxiomC{$\Gamma \Ra \neg \Delta$} \LeftLabel{(\texttt{ConCp})}
		\UnaryInfC{$  \Delta \Ra \neg \Gamma$}
		\DisplayProof
		
		&
		
		\AxiomC{$\neg \Gamma \Ra \Delta$}
		\LeftLabel{(\texttt{ClCp})}
		\UnaryInfC{$\neg \Delta \Ra \Gamma$}
		\DisplayProof
		\\
	& \\
	
			\AxiomC{$A(t) , \Gamma \Ra \Delta$} 
	\LeftLabel{(\texttt{L}$\forall$ )}
	\UnaryInfC{$\forall x A, \Gamma \Ra \Delta$}
	\DisplayProof
	&
		\AxiomC{$\Gamma \Ra  \Delta, A(y)$} 
	\LeftLabel{(\texttt{R}$\forall$)}
	\UnaryInfC{$\Gamma \Ra \Delta,  \forall x  A $}\noLine
	\UnaryInfC{$\text{$y \notin {\tt FV}(\Gamma, \Delta, \forall x A)$}$}
	\DisplayProof
	
		\\ & \\
	\end{tabular}

\end{center}

We write ${\tt rk}(A)$ for the logical complexity of $A$, defined as the number of nodes in the longest branch of its syntactic tree. For a derivation $d$ we let
\begin{itemize}
    \item ${\tt hgt}(d) := {\tt sup}_{i<n} \{{\tt hgt} (d_i) + 1 \sth \text{ $d_i$ an immediate subderivation of $d$}\}$ (the height of the derivation), where $d_0, ..., d_{n}$ are the immediate subderivations of $d$ (the \emph{cut-rank} of $d$).
\end{itemize}
We say that a formula $A$ is derivable in a system, if the sequent $\Ra A$ is derivable in it. 

The next lemma collects some basic facts about ${\bf  G1h_{cd}}$. They mostly concern the admissibility of some basic inferences in ${\bf  G1h_{cd}}$.
\begin{lemma}\hfill \label{lem:baseprop}
    \begin{enumeratei}\setlength\itemsep{1ex}
        \item  The sequents $\Ra \top$, $A \Ra \neg \neg A$,  $\neg \neg A \Ra A$, are derivable in  ${\bf  G1h_{cd}}$.
        \item The rule of contraposition
            \begin{center}
                 \AXC{$\Gamma \Ra \Delta$}
                     \UIC{$\neg \Delta \Ra \neg\Gamma$}
                 \DP
             \end{center}
             is admissible in ${\bf  G1h_{cd}}$.
    \item The following rules are admissible in ${\bf  G1h_{cd}}$:
    
    	\begin{align*}
	    &\AxiomC{$ A , B, \Gamma \Ra \Delta $} 
		\LeftLabel{${\tt ( L \wedge ) }$} 
		\UnaryInfC{$ A \wedge B, \Gamma \Ra  \Delta$}
		\DisplayProof 
		&&	\AxiomC{$ \Gamma \Ra  A, \Delta $}
		\AxiomC{$ \Gamma \Ra  B, \Delta $}
		\LeftLabel{${ \tt (R\wedge ) }$}
		\BinaryInfC{$ \Gamma \Ra  A \wedge B, \Delta $}
		\DisplayProof\\[1em]
        &	\AxiomC{$A(y) , \Gamma \Ra \Delta$} 
	\LeftLabel{${(\tt L \exists)}$}
	\RightLabel{$y \notin {\tt FV} (\Gamma, \Delta, \exists x A)$}
	\UnaryInfC{$\exists x A, \Gamma \Ra \Delta$}
	\DisplayProof
	        &&\AxiomC{$\Gamma \Ra  \Delta, A(t)$} 
	\LeftLabel{$({\tt R}\exists)$}
	\UnaryInfC{$\Gamma \Ra \Delta,  \exists x  A $}\noLine
	\DisplayProof
	\end{align*}
	
    \item Intersubstitutivity: If $\chi \Ra \chi'$ and $\chi' \Ra \chi$, as well as $\psi$ are derivable in ${\bf  G1h_{cd}}$, then $\psi (\chi' / \chi)$ is derivable, where $\psi (\chi' / \chi)$ is obtained by replacing all occurrences of $\chi$ in $\psi$ by $\chi'$. 
    \end{enumeratei}
\end{lemma}

\begin{proof}
    Claims (i)-(iii) are direct consequences of the contraposition rules (\texttt{ConCp}) and (\texttt{ClCp}). (iv) is proved by a straightforward induction on the height of the derivation in ${\bf  G1h_{cd}}$. 
\end{proof}

We opted for this specific formulation of ${\bf  G1h_{cd}}$ mainly because it substantially simplifies the presentation of the results of the next sections, which are the main focus of the paper. From a proof-theoretic point of view, the calculus has some drawbacks even at the propositional level, as the rules {\tt ConCp} and {\tt ClCp} compromise the induction needed for cut-elimination. In the propositional case, even if one removes {\tt ConCp} and {\tt ClCp} and splits the contraposition rule of Lemma \ref{lem:baseprop}(ii) on a case by case manner, problems for cut-elimination remain\cite{fis20}. Moreover, when one moves to the quantificational system, there are deeper problems.  The same counterexample that is employed to show that cut is not admissible in systems of intuitionistic logic with constant domains can be employed for the systems we are investigating.\footnote{See for example L\'opez-Escobar \cite{lop83}.} Both problems can be addressed by employing techniques from Kashima and Shimura \cite{kash94}, which however rely on the extension of the systems with additional resources.

Since cut elimination is not the main focus of our paper, we opt for a more compact presentation of ${\bf  G1h_{cd}}$ that fits nicely our purpose of extending it with arithmetic and truth rules.

\subsection{Semantics}
In this section we present the semantics of ${\bf  G1h_{cd}}$ (and therefore of ${\bf HYPE}$) and sketch its completeness with respect to it. We follow a simplification of the semantics in Leitgeb \cite{lei19} suggested by Speranski \cite{spe20}. Speranski connects ${\bf HYPE}$-models with Routley semantics. A Routley frame $\mathfrak{F}$ is a triple $\langle W, \leq, * \rangle$, where:
\begin{enumeratei}
    \item $W$ is a non-empty set of states;
    \item $\leq$ is a preorder;
    \item $*$ is a function from $W$ to $W$, which is:
    \begin{itemize-}
        \item \emph{antimonotone}, i.e.~for all $w,v \in W$, if $w \leq v$, then $v^* \leq w^*$;
        \item \emph{involutive}, i.e.~for all $w \in W$, $w^{**} = w$.
    \end{itemize-}
\end{enumeratei}
A constant domain model $\mathfrak{M}$ for ${\bf HYPE}$ is a triple $(\mathfrak{F}, {D},I)$ where $\mathfrak{F}$ is a Routley frame, ${D}$ is a non-empty set (the domain of the model), and $I$ is an interpretation function. In particular, $I$ assigns to every constant $c$ an element of $D$ and it associates with each state $w$ and $n$-place predicate $P$ a set $P^{w} \subseteq D^n$. Constants are interpreted rigidly and, although domains do not grow, we impose the following hereditariness condition:
for all $v,w \in W$, if $v \leq w$, then for all predicates $P$, $P^{v} \subseteq P^{w}$.

Let $\mathfrak{M}$ be a constant domain model, $w \in W$ and $\sigma \colon {\rm VAR}\to D$ a variable assignment on $D$, then the forcing relation $\mathfrak{M}, w, \sigma \Vdash A$ is defined inductively:
\begin{flalign*}
& \mathfrak{M}, w, \sigma \Vdash P (x_1,...,x_n) && \text{ iff }    (\sigma (x_1),..., \sigma(x_n)) \in P^{w}; & & \\   
 & \mathfrak{M}, w, \sigma \Vdash \neg  A && \text{ iff }   \mathfrak{M}, w^*, \sigma \nVdash A; \\
 & \mathfrak{M}, w, \sigma \Vdash A \vee B && \text{ iff }   \mathfrak{M}, w, \sigma \Vdash A \text{ or } \mathfrak{M}, w, \sigma{} \Vdash B; & & \\ 
 &  \mathfrak{M}, w, \sigma \Vdash A \rightarrow B && \text{ iff }   \text{ for all } v, \text{ with } w \leq v \text{, if  } \mathfrak{M}, v, \sigma \Vdash A  \text{, then  } \mathfrak{M}, v, \sigma \Vdash B; & & \\   
& \mathfrak{M}, w, \sigma \Vdash \forall x A && \text{ iff }  \text{ for all $x$-variants $\sigma'$ of $\sigma$},  \mathfrak{M}, w, \sigma' \Vdash A;  & & \\  
&\mathfrak{M}, w, \sigma \nVdash \bot.  
\end{flalign*}
Finally, we define logical consequence. We write, for $\Gamma,\Delta$ sets of sentences:
\begin{itemize}
    \item $\mf{M},w\Vdash \Gamma \Ra \Delta$ iff: if $\mf{M},w\Vdash \gamma$ for all $\gamma\in \Gamma$, then $\mf{M},w \Vdash \delta$ for some $\delta\in \Delta$;
    \item $\Gamma \Vdash \Delta$ iff for all $\mf{M},w$:   $\mf{M},w\Vdash \Gamma \Ra \Delta$.
\end{itemize}

The system ${\bf  G1h_{cd}}$ is equivalent to the following Hilbert-style system ${\bf QN^\circ}$ featuring the axiom schemata
\begin{align*}
        &A\ra (B\ra A) &&A\ra (B\ra C)\ra ((A\ra B)\ra (A\ra C))\\
        &A\land B\ra A && A\land B\ra B\\
        &A\ra A\vee B &&B\ra A\vee B\\
        &A\ra (B\ra A\land B) && (A\ra C) \ra ((B\ra C)\ra (A\vee B\ra C))\\
       & \neg\neg A\ra A && A \ra \neg \neg A\\
        &\forall x A \ra A(t)&& A(t)\ra \exists x A
    \end{align*}
and the following rules of inference:

\begin{align*}
    &\AXC{$A$}
     \AXC{$A \ra B$}\RightLabel{(MP)}
        \BIC{$ B $}
    \DP
    &&
        \AXC{$A \ra B$}\RightLabel{(CP)}
            \UIC{$\neg B \ra \neg A$}
        \DP\\[1em]
 & \AXC{$A \ra B(x)$}
\RightLabel{$x$ not free in $A$}
\UIC{$A \ra \forall x B$}
\DP
 &&
    \AXC{$A(x) \ra B$}
        \RightLabel{$x$ not free in $B$}
        \UIC{$\exists x A \ra B$}
    \DP\\
\end{align*}
${\bf QN^\circ}$ is a neater presentation of ${\bf HYPE}$ where a few redundant principles are dropped. The consequences of the two systems are identical.

That our system ${\bf  G1h_{cd}}$ is equivalent to ${\bf QN}^\circ$ can be seen as follows. ${\bf  G1h_{cd}}$ is an extension of intuitionistic logic (modulo the definition of $\sim A$ as $A\ra \bot)$. Therefore, since all axioms of   ${\bf QN^\circ}$ except for the double negation axioms are intuitionistically valid, Lemma \ref{lem:baseprop} enables us to show that all axioms of ${\bf QN^\circ}$ are consequences of  ${\bf  G1h_{cd}}$. Additionally, Lemma \ref{lem:baseprop} shows that contraposition is admissible in ${\bf  G1h_{cd}}$. Rules for quantifiers are easily established in ${\bf  G1h_{cd}}$. For the other direction a proof on the length of the derivation is sufficient. The fact that  the deduction theorem holds in ${\bf QN^\circ}$ renders the proof particularly simple. Therefore, we have:
\begin{lemma} \label{lem:equiv}
 ${\bf  G1h_{cd}} \vdash \Gamma \Ra \Delta$ iff ${\bf QN^\circ} \vdash \bigwedge \Gamma \ra \bigvee \Delta$.
\end{lemma}
\noindent Lemma \ref{lem:equiv} then entails that ${\bf  G1h_{cd}}$ is equivalent to Leitgeb's ${\bf HYPE}$.

Speranski \cite{spe20} establishes a strong completeness result (for countable signatures) for ${\bf QN}^\circ$. Speranski uses a Henkin-style proof similar to the strategy employed in Gabbay et al. \cite[\S7.2]{gaea09} for intuitionistic logic with constant domains.  Leitgeb \cite{lei19} establishes a (weak) completeness proof for his Hilbert style system based on the work of G\"ornemann \cite{goe71}. By Lemma \ref{lem:equiv} we can employ Speranski's completeness result for our system ${\bf  G1h_{cd}}$ with respect to Routley semantics:

\begin{proposition}[Completeness of ${\bf  G1h_{cd}}$ \cite{spe20}]
 $\Gamma \Vdash \Delta$ iff there is a finite $\Delta_0 \subseteq \Delta$, such that $\Gamma \vdash_{\bf QN^\circ} \Delta_0$.
\end{proposition}
We now turn to investigating how much classical reasoning can be reproduced in our logic. Such questions will turn out to be essential components of the analysis of truth theories over ${\bf HYPE}$.

\subsection{HYPE and recapture}

One of the desirable properties of the nonclassical logics employed in the debate on semantic paradoxes is the capability of {recapturing} classical reasoning in domains where there is no risk of paradoxicality, such as mathematics -- see e.g.~\cite{fie08}.\footnote{This form of recapture is a slightly different phenomenon from a direct, provability preserving, translation of the entire language of one theory in the other, as it happens for instance in the famous G\"odel-Gentzen translation or the S4 interpretations of classical in intuitionistic logic, or intutionistic logic in modal logic respectively. While those translations provide a method to reinterpret the logical vocabulary -- by keeping the non-logical vocabulary fixed -- in a provability-preserving way, recapture strategies typically show that, for a specific fragment of its language, the nonclassical theory behaves according to the rules of classical logic. For instance, that a nonclassical theory of truth behaves fully classically if one restricts her attention to the truth-free language. {To carry on with the analogy with the relationships between classical and intuitionistic logic, recapture strategies are much closer to the identity between the $\Delta_1$-fragments of classical and intuitionistic arithmetic.} }

The following lemma summarizes the recapture properties of ${\bf  G1h_{cd}}$ and extensions thereof. It essentially states that, in systems based on ${\bf  G1h_{cd}}$, once we restrict our attention to a fragment of the language satisfying the excluded middle and/or explosion, the native {\bf HYPE}-negation and conditional, as well as the defined intuitionistic negation, all behave classically. 

\begin{lemma}\label{lem:rechyp}\hfill
  \begin{enumeratei}\setlength\itemsep{1em}
    \item The following rules are admissible in extensions of ${\bf  G1h_{cd}}$:
    
    \begin{align*}    
	    &\AxiomC{$\Ra A, \neg A$}
	        \AxiomC{$\Gamma, A \Ra \Delta$}
	            \BinaryInfC{$\Gamma \Ra \neg A, \Delta$}
    	\DisplayProof
	    &&
	    \AxiomC{$A, \neg A \Ra$}
        	\AxiomC{$\Gamma\Ra A,\Delta$}
            	\BinaryInfC{$\Gamma, \neg A \Ra \Delta$}
    	\DisplayProof\\[1ex]
    	&
             \AXC{$\Ra A,\neg A$}
                    \AXC{$\Gamma,A\Ra B,\Delta$}
                        \BIC{$\Gamma \Ra A\ra B,\Delta$}
                    \DP
  \\[1ex]
        &  \AxiomC{$A, \neg A \Ra$}
            \UIC{$\neg A\Ra A\ra \bot$}
                \DP
         &&  \AxiomC{$A, \neg A \Ra$}
            \UIC{$ A\ra \bot \Ra \neg A$}
                \DP
       \\[1ex]
                 &\AXC{$\Ra A,\neg A$}
                    \UIC{$A \ra B \Ra A \supset B$}  
                \DP
                & &     \AXC{$\Ra A,\neg A$}
                    \UIC{$A \supset B \Ra A \ra B$}
                \DP\\
     \end{align*}
     
     \item The previous fact can be used to show, by an induction on ${\tt rk}(A)$, that $\Ra A,\neg A$ is derivable for any formula whenever $\Ra P,\neg P$ is derivable for any atomic $P$ in $A$.
    \end{enumeratei}
\end{lemma}

\begin{proof}
We prove the claims for the crucial cases in which a conditional is involved:

For (i):

 {\footnotesize\begin{center}
	\AXC{$\Ra A, \neg A$}
	\UIC{$\Gamma \Ra A, \neg A, B, \Delta$}
	\AXC{$\Gamma, A \Ra B, \Delta$}
	\UIC{$\Gamma, A \Ra \neg A, B, \Delta$}
	\BIC{$\Gamma \Ra \neg A, B, \Delta$}
	\UIC{$\Gamma \Ra B, \neg A, A \ra B, \Delta$}
	\AXC{$B,A \Ra B$}
	\UIC{$B \Ra A \ra B$}
	\UIC{$\Gamma, B \Ra  \neg A, A \ra B, \Delta$}
	\BIC{$\Gamma \Ra \neg A, A \rightarrow B,  \Delta$}
	\AXC{$A, \neg A \Ra$}
	\UIC{$A, \neg A \Ra B$}
	\UIC{$\neg A \Ra A \ra B$}
	\UIC{$\Gamma, \neg A \Ra A \ra B, \Delta$}
	\BIC{$\Gamma \Ra A \ra B, \Delta$}
	\DP
 \end{center}}
 
 For (ii): 
\begin{center}

	\AXC{$\neg A, A \Ra B$}
	\UIC{$\neg A \Ra A \ra B$}
	\UIC{$\neg (A \ra B)  \Ra A$}
	\AXC{$B, A \Ra B$}
	\UIC{$B \Ra A \ra B$}
	\UIC{$\neg (A \ra B)  \Ra \neg B$}
	\UIC{$B, \neg (A \ra B)  \Ra $}
	\BinaryInfC{$A \ra B, \neg (A \ra B) \Ra$}
	\DP
	
\end{center}
\end{proof}

\begin{remark}
   The induction involved in Lemma \ref{lem:rechyp}(ii) does not go through in intuitionistic logic with the ${\bf HYPE}$-negation $\neg$ replaced by the intuitionistic negation $\sim$. 
\end{remark}





\subsection{Equality}

For our purposes it's important to extend ${\bf  G1h_{cd}}$ a theory of equality. ${\bf  G1h_{cd}^=}$ is obtained by adding to ${\bf  G1h_{cd}}$ the following initial sequents for equality. 
\begin{align*}
   \tag{\tt Ref} &  \Ra t = t  \\
   \tag{\tt Rep} & s = t, A (s) \Ra A (t)
\end{align*}

By an essential use of ${\tt ConCp}$, we can establish in ${\bf  G1h_{cd}^=}$ that identity statements behave classically.
\begin{lemma}\label{lem:eqfde}
 ${\bf  G1h_{cd}^=}$ derives $\Ra s = t, \neg s = t$ and  $s = t, \neg s = t \Ra$.
\end{lemma}
\begin{proof} We use the identity sequents:

\begin{center}

	\AXC{$s = t, \neg s = t \Ra \neg t = t$}
	\AXC{$\Ra t = t$}
	\UIC{$\neg t = t \Ra$}
	\BIC{$s = t, \neg s = t \Ra$}
	\UIC{$\Ra \neg s = t, \neg \neg s = t$} 
	\UIC{$\Ra \neg s = t, s = t$} 
	\DP
	
	\end{center}
\end{proof}
Lemma \ref{lem:eqfde} reveals some subtle issues concerning the treatment of identity in subclassical logics generally employed to deal with semantical paradoxes. It tells us that identity is essentially treated as a classical notion in ${\bf  G1h_{cd}^=}$. To obtain a similar phenomenon in absence of  ${\tt ConCp}$ and ${\tt ClCp}$, one would have to add the counterpositives of ${\tt Rep}$ and ${\tt Ref}$ to the system. A nonclassical treatment of identity would require some non-trivial changes to ${\tt Rep}$ and ${\tt Ref}$. That identity is a classical notion is perfectly in line with our framework, in which identity is a non-semantic notion akin to mathematical notions.

\section{Arithmetic in {\tt HYPE}}\label{sec:hypari}

Starting with the logical constants introduced above and the identity symbol, we now work with a suitable expansion of the usual signature $\{0,{\tt S},+,\times\}$ by finitely many function symbols for selected primitive recursive functions. Such function symbols are needed for a smooth representation of  formal syntax. We call this language $\lnat^\ra$. We will also make use of the $\ra$-free fragment of the language of arithmetic, which we label as $\lnat$. Our base theory will then be obtained by adding, to the basic axioms for $0,{\tt S},+,\times$ (axioms ${\bf Q}1$-$2$, ${\bf Q}4$-$7$ of \cite{hapu93}), the recursive clauses for these additional function symbols. The resulting system will be called ${\bf HYA}^-$.

In the following, the role of rule and axiom schemata will be crucial. It will be particularly important to keep track of the classes of instances of a particular schema, and therefore we will always relativize schemata to specific languages and understand the schema as the set of all its instances in that language.  For example, in the case of the induction axioms we use the label ${\tt IND}^\rightarrow (\mathcal{L})$ to refer to the set of all sequents of the form 
\[ \tag{${\tt IND}^\rightarrow(\mc{L})$} \Ra A (0) \wedge \forall x (A(x) \rightarrow A(x+1)) \rightarrow \forall x A(x), \]
 where $A$ is a formula of $\mathcal{L}$. Similarly, induction rules ${\tt IND^R(\mc{L})}$ will refer to all rule instances
 
 \medskip
\begin{center}
\AxiomC{$ \Gamma, A(x) \Ra A(x+1), \Delta $}
\RightLabel{${\tt (IND^R(\mc{L}))}$}
\UnaryInfC{$ \Gamma, A (0) \Ra A(t), \Delta $}
\DisplayProof
\end{center}
for  $A$ a formula of $\mathcal{L}$.

We call ${\bf HYA}$ the extension of ${\bf HYA}^-$ by the induction axiom ${\tt IND}^\rightarrow(\mc{L})$. ${\bf HYA}$ is equivalent to Peano Arithmetic ${\bf PA}$. This is essentially because of the recapture properties of our logic. For formulas $A$ containing only classical vocabulary, the properties stated in Lemma \ref{lem:rechyp} entail that the rule and sequent formulations of induction are equivalent.
\begin{lemma} \label{lem:classid}
Let $\mc{L}\supseteq \lnat^\ra$. Over ${\bf HYA}^-$: ${\tt IND^R}(\mc{L})$ and ${\tt IND}^\rightarrow(\mc{L})$ are equivalent when restricted to formulas $A$ such that $\Ra A, \neg A$.
\end{lemma}

Since for $A \in \lnat^\ra$, $\Ra A, \neg A$ and $A, \neg A \Ra$ are derivable in ${\bf G1h_{cd}}$, we have the immediate corollary that:

\begin{corollary} \label{corol1} ${\bf HYA}$ is equivalent to ${\bf PA}$.
\end{corollary}

\subsection{Ordinals and transfinite induction}

Our notational conventions for schemata generalize to sche\-ma\-ta other than induction. A prominent role in the paper will be played by \emph{transfinite induction schemata}. In order to introduce them, we need to assume a notation system $({\tt OT},\prec)$ for ordinals up to the Feferman-Sch\"utte ordinal $\Gamma_0$ as it can be found, for instance, in \cite[Ch.~2]{poh09}. ${\tt OT}$ is a primitive recursive set of ordinal codes and $\prec$ a primitive recursive relation on ${\tt OT}$ that is isomorphic to the usual ordering of ordinals up to $\Gamma_0$. We distinguish between fixed ordinal codes, which we denote with $\alpha,\beta,\gamma\ldots$, and $\zeta,\eta,\theta,\xi,\ldots $ as abbreviations for variables ranging over elements of {\tt OT}. Our representation of ordinals satisfies all standard properties. In particular, we will make implicit use of the properties listed in \cite{trsc00}, p.~322.


We will make extensive use of the following abbreviations. We call a formula \emph{progressive} if it is preserved upwards by the ordinals: 
\[
   { \tt Prog}(A) := \forall \eta(\forall\zeta\prec\eta \, A(\zeta)\ra A (\eta))
\]
where $\forall\zeta\prec\eta \, A(\zeta)$ is short for $\forall\zeta(\zeta\prec\eta\ra A(\zeta))$. We will use this (standard) notational convention in several occasions in what follows. Similarly, we will write $\exists \zeta \prec \eta\, A(\zeta)$ for $\exists \zeta(\zeta\prec \eta\land A(\zeta))$.

This formulation of progressiveness is  ${\bf HYA}$-equivalent to a formulation as a sequent $\forall\zeta\prec\eta \, A(\zeta) \Ra A (\eta)$. Moreover, if $A(x)\vee \neg A(x)$ is provable, then ${\tt Prog}(A)$ is ${\bf HYA}$-equivalent to:
\begin{equation}
    \forall \eta(\forall\zeta\prec\eta \, A(\zeta)\supset A (\eta)).
\end{equation}

Transfinite induction up to the ordinal $\alpha \; (\prec \Gamma_0)$ will be formulated as the following sequent:
     \[ \tag{${\tt TI_\alpha}(A)$} {\tt Prog} (A) \Ra \forall \xi \prec \alpha \, A(\xi) \]  

 An alternative would be to use a rule-formulation:  
\[    
    \AxiomC{$\Gamma, \forall \zeta \prec \eta \, A(\zeta) \Ra A(\eta)$} 
	\LeftLabel{${\tt TI_\alpha^r}(A) :=$}
	\UnaryInfC{$\Gamma \Ra\forall\xi\prec\alpha\,A(\xi), \Delta$}
	\DisplayProof
\]
${\tt TI_\alpha^r}(A)$ differs from the standard rule formulation of transfinite induction (see, e.g.~\cite{hal14}) in that its premiss features only one formula in the succedent.

The two formulations of induction just introduced are equivalent over ${\bf HYA}^-$, i.e.~given ${\tt TI_\alpha}(A)$, ${\tt TI_\alpha^r}(A)$ is admissible, and given ${\tt TI_\alpha^r}(A)$, ${\tt TI_\alpha}(A)$ is derivable.\footnote{The notion of admissible rule that we employ is the one from \cite[p.~76]{trsc00}.}  

${\tt TI}_\alpha(\mc{L})$ is short for ${\tt TI}_\alpha (A)$ $\mathrm{for\,every\, formula}\, A\, \mathrm{of\,the\,language}$ $\mc{L}$. ${\tt TI}_{< \alpha} (\mc{L})$ is short for ${\tt TI}_\beta (\mc{L})$ for all $\beta \prec \alpha$. The function $\omega_n$ is recursively defined in the standard way as: $\omega_0=1$, $\omega_{n+1}=\omega^{\omega_n}$. 
 


%

\subsection{Transfinite induction and nonclassical predicates}

Our main purpose in this paper is to study the proof-theoretic properties of extensions of ${\bf HYA}$ with additional predicates that may not behave classically -- i.e.~they may not satisfy Lemma \ref{lem:classid}. In fact, in the case of the pure arithmetical language, Lemma \ref{lem:classid} gives us immediately that ${\bf HYA}$ derives ${\tt TI}_{< \varepsilon_0} ( \lnat^\ra )$.  In this section we show directly that Gentzen's original proof of ${\tt TI}_{< \varepsilon_0} ( \lnat^\ra )$ can be carried out in ${\bf HYA}$ for suitable extensions of $\lnat^\ra$.\footnote{Troelstra \& Schwichtenberg \cite{trsc00} already established that the Gentzen proof can be carried out in the minimal $\rightarrow \forall \bot$ fragment of ${\tt IL}$.}

\begin{theorem} \label{thm:genhyp}
Let $\mathcal{L}^+$ be a language expansion of $\lnat^\ra$ by finitely many predicate symbols. Then 
${\bf HYA} \vdash {\tt TI}_{< \varepsilon_0} (\mathcal{L}^+)$. 
\end{theorem}
\noindent The rest of this subsection will be devoted to the proof of Theorem \ref{thm:genhyp}, which will involve several preliminary lemmata.


A key ingredient of Gentzen's proof -- which will also play an important role in subsequent sections -- is Gentzen's jump formula:
 \[
    A^+(\theta) := \forall\xi(\forall\eta(\eta\prec\xi\rightarrow A(\eta))\rightarrow \forall\eta(\eta\prec\xi+\omega^\theta\rightarrow A(\eta))).
\]

\begin{lemma}\label{lemma1}
For any $A\in \mc{L}^+$, ${\bf HYA}$ proves ${\tt Prog} (A) \Ra {\tt Prog}(A^+)$ . 
\end{lemma}

\begin{proof}


The informal argument is as follows:
We assume ${\tt Prog}(A)$ and we want to show  ${\tt Prog}(A^+)$, i.e.~$\forall \zeta \prec \theta \, A^+ (\zeta) \ra A^+ (\theta)$. 
So we also assume $\forall \zeta \prec \theta \, A^+ (\zeta)$ and $\forall\zeta(\zeta\prec\xi\rightarrow A(\zeta))$ and $\eta\prec\xi + \omega^\theta$ to show $A(\eta)$.

Informally, we make a case distinction: Either $\theta = 0$ or $\theta \succ 0$. 

 \emph{Case 1}: If  $\theta = 0$, then 
   \begin{equation}\label{eq:tic01}
        \theta = 0, \eta\prec\xi+\omega^\theta\Ra\eta\prec\xi\,\vee\,\eta=\xi.
    \end{equation}
    We have, by the reflexivity sequents and logical rules:
    \begin{align}
        &\forall\zeta \,(\zeta\prec\xi\ra A (\zeta)), \eta\prec\xi \Ra A (\eta)
    \end{align}
    Again by reflexivity and the identity axioms:
    \begin{align}
        &{\tt Prog}(A), \forall\zeta\,(\zeta\prec\xi\ra A (\zeta)), \eta=\xi \Ra A (\eta). 
    \end{align}
By \eqref{eq:tic01} and {\tt Cut}, we obtain
    \begin{equation}\label{eq:tic05}
    \theta = 0, {\tt Prog}(A), \forall\zeta\,(\zeta\prec\xi\ra A (\zeta)), \eta\prec\xi+\omega^\theta \Ra A(\eta).
    \end{equation}
   
     \emph{Case 2}: $\theta \succ0$.
    Then by a derivable version of Cantor's Normal Form Theorem: 
\[ \tag{$\dagger$} \theta \succ 0,  \eta\prec\xi+\omega^\theta\Ra \exists n \,\exists \theta_0\prec \theta (\eta\prec\xi+\omega^{\theta_0}\cdot n). \]

Given that induction for ordinal notations up to $\omega$ is provable in ${\bf HYA}$, we will show by induction on $n\prec\omega$ that 
\[ \forall \zeta \prec \, \theta \, A^+ (\zeta),\theta_0 \prec \theta 
    \Ra \forall \zeta (\zeta\prec \xi + \omega^{\theta_0} \cdot n \rightarrow  A (\zeta)). \]
The base case is straightforward because the following is trivially derivable (by property ({\tt ord6})):
    \begin{equation}
        \forall \eta \prec \xi \, A (\eta)  \Ra (\forall \eta\prec\xi + \omega^{\theta_0}\cdot 0) \, A(\eta).
    \end{equation}
    For the induction step, we start by noticing that by instantiating $\xi$ in $A^+(\theta_0)$ with $\xi+\omega^{\theta_0}\cdot n$, we obtain:
    \begin{equation}
        A^+(\theta_0) \Ra \forall \zeta \prec \xi + \omega^{\theta_0} \cdot n \, A (\zeta)  \ra \forall \zeta \prec \xi + \omega^{\theta_0} \cdot (n +1) \, A (\zeta),
    \end{equation}
    As mentioned, by letting:
    \[
        B(x):= \forall \zeta\prec \xi +\omega^{\theta_0}\cdot x \,A(\zeta)
    \]
    $\textbf{HYA}$ proves the $\omega$-induction principle (with $n\prec \omega$):
    \[
    B(0), \forall n (B(n)\ra B(n+1)) \Ra \forall n\, B(n).
    \]
    Therefore, by a series of cuts, we obtain:
    \begin{equation}\label{eq:tihypea}
        A^+(\theta_0), \forall \zeta \prec \xi \, A (\zeta)  \Ra \forall n \forall \zeta \prec \xi + \omega^{\theta_0} \cdot n \, A (\zeta).
    \end{equation}
    
   From \eqref{eq:tihypea} we obtain:
    \begin{equation}
    \forall \zeta \prec \xi \, A (\zeta), \forall\zeta\prec\theta \, A^+(\zeta), \theta_0\prec\theta \Ra \forall n \forall \zeta \prec \xi + \omega^{\theta_0} \cdot n \, A (\zeta).
    \end{equation}
 Therefore, we can instantiate $n$ and $\zeta$ (with $\eta$), and move the antecedent of $\eta \prec \xi + \omega^{\theta_0} \cdot n \ra A (\eta)$ from the right-hand side to the left hand side of the sequent arrow. Since both $n$ and $\eta$ are general, we can existentially generalize over them to get:
    \begin{equation}
            \theta \succ 0, \forall \zeta \prec \xi \, A (\zeta), \forall\zeta\prec\theta \, A^+(\zeta), \exists n \,\exists \theta_0\prec\theta(\eta\prec\xi+\omega^{\theta_0}\cdot n) \Ra  A (\eta),
    \end{equation}
    which in turn by $(\dagger)$ gives us:
    \begin{equation}
        \theta \succ 0, \forall \zeta \prec \xi \, A (\zeta), \forall\zeta\prec\theta \, A^+(\zeta), \eta\prec\xi+\omega^\theta \Ra A (\eta).
    \end{equation}
    Now we combine the two cases. Together with our (\ref{eq:tic05}) in Case 1, the last sequent enable us to derive: 
    \[ \theta = 0 \vee \theta \succ 0,{\tt Prog}(A),  \forall \zeta \prec \xi \, A (\zeta), \forall\zeta\prec\theta \, A^+(\zeta), \eta\prec\xi+\omega^\theta \Ra A (\eta). \]
    By the provability of $\theta = 0 \vee \theta \succ 0$ and applications of the rules $({\tt R} \rightarrow)$ and $({\tt R} \forall)$ we finally get 
    \[ {\tt Prog}(A) \Ra {\tt Prog}(A^+). \]     
\end{proof}
The progressiveness of Gentzen's jump formula enables us then to establish:

\begin{lemma} \label{lemmatij}
If ${\tt TI}_\alpha (\mathcal{L}^+)$ is derivable in $\bf{HYA}$, then ${ \tt TI}_{\omega^\alpha}({\mc L}^+)$ is derivable in $\bf{HYA}$.
\end{lemma}

\begin{proof}
We assume ${\tt TI}_\alpha (\mathcal{L}^+)$. Specifically we have 
\begin{equation}
    {\tt Prog}(A^+) \Ra \forall\xi\prec\alpha \,A^+(\xi).
\end{equation}
By the meaning of ${\tt Prog}(A^+)$, we obtain
\begin{equation}
    {\tt Prog}(A^+) \Ra A^+(\alpha).
\end{equation}
By the previous Lemma \ref{lemma1} and cut we also have
\begin{equation}
    {\tt Prog}(A) \Ra  A^+(\alpha).
\end{equation}

which is
\begin{equation}\label{eq:aplusfin}
   {\tt Prog}(A)  \Ra \forall\xi(\forall\eta \prec\xi \, A(\eta) \rightarrow \forall\eta \prec\xi+\omega^\alpha \, A(\eta)).
\end{equation}

But also
\begin{equation}
   \Ra \forall \eta\prec 0 \,A(\eta),
\end{equation}
and therefore by \eqref{eq:aplusfin} taking $\xi = 0$, we obtain


    \[
       {\tt Prog}(A)  \Ra \forall\eta\prec\omega^\alpha A(\eta),
    \]
    as desired. 

\end{proof}

\begin{corollary} \label{cor:clti}
    If $A$ is such that ${\bf HYA}$ proves $A(x)\vee \neg A(x)$, we have that, if ${\bf HYA}$ proves the classical transfinite induction axiom schema for $\alpha$
    \[
        (\forall \zeta\prec \eta A(\zeta)\supset A(\eta))\supset \forall \xi\prec \alpha \,A(\xi),
    \]
    then ${\bf HYA}$ proves:
    \[
        (\forall \zeta\prec \eta A(\zeta)\supset A(\eta))\supset \forall \xi\prec \omega^\alpha \,A(\xi).
    \]
\end{corollary}

All is set up to finally prove the main result of this section, the admissibility in ${\bf HYA}$ of the required schema of transfinite induction up to any ordinal $\alpha \prec \varepsilon_0$. 

\begin{proof}[Proof of Theorem \ref{thm:genhyp}]
    The result follows immediately from Lemma \ref{lemmatij}. Since ${\tt TI}_{\omega_0}(A)$ is trivially derivable in ${\bf HYA}$, the lemma tells us that ${\tt TI}_{\omega_n}(A)$, for each $n$, can be reached in finitely many proof steps. 
 \end{proof}
Theorem \ref{thm:genhyp} is key to our proof-theoretic analysis of a theory of truth over ${\bf HYPE}$. We now turn to the definition of such a truth theory.

\section{The Theory of Truth $\kfl$}

In this section we introduce the theory of truth $\kfl$, standing for Kripke-Feferman-Leitgeb. The theory is formulated in the language $\lt^\ra:=\lnat^\ra \cup \{\T\}$, where $\T$ is a unary predicate for truth. $\kfl$ is a theory of truth for a $\ra$-free language $\lt$, which is simply the $\ra$-free fragment of $\lt^\ra$. In $\kfl$, the conditional $\ra$ should be thought of as a theoretical device to articulate our semantic theory, and not as an object of semantic investigation. We elaborate on the role of the conditional in the concluding section \ref{sec:conc}. Semantically (cf.~\S\ref{sect:sem}), the conditional amounts to a device to navigate between fixed point models of $\lt$ in the sense of \cite{kri75}.

\begin{definition}[The language $\lt$]
    The logical symbols of $\lt$ are $\bot, \neg,\vee,\forall$. In addition, we have the identity symbol $=$. Its non-logical vocabulary amounts to the arithmetical vocabulary of $\lnat$ and the truth predicate $\T$. 
\end{definition}

We assume a canonical representation of the syntax of $\lt$  in ${\bf HYA}$. Given the equivalence of ${\bf HYA}$ and ${\bf PA}$ for arithmetical vocabulary stated in Corollary \ref{corol1}, we can assume one of the standard ways of achieving this (e.g.~\cite{can89}).  We apply most of the notational conventions -- e.g.~Feferman's dot notation -- described in \cite[\S I.5]{hal14}.
\begin{definition}[The theory $\kfl$]
 $\kfl$ extends ${\bf HYA}$ formulated in $\lt^ \ra$ -- i.e.~with the induction schema extended to $\lt^\ra$ -- with the following truth initial sequents:
    \begin{align}
        \tag{$\kfl$1}  {\ct} (x) \wedge {\ct} (y) &\Ra \T (x \puu{=} y) \leftrightarrow {\tt val}(x) = {\tt val}(y)  \\
         \tag{$\kfl$2}  &\Ra  \T ( \ulc{\T} \dot{x}\urc) \leftrightarrow \T \, x  \\
        \tag{$\kfl$3}  \sentkf(x) &\Ra \T\puu{\neg}x\lra \neg \T \, x\\
       \tag{$\kfl$4} \sentkf(x) \wedge \sentkf(y) & \Ra  \T (x \puu{\vee} y) \leftrightarrow \T \,x \vee \T \,y \\
         \tag{$\kfl$5} \sentkf(\puu{\forall} v x ) \wedge {\tt var}(v )&\Ra  \T (\puu{\forall} v x ) \leftrightarrow \forall y({\tt CTerm}_{\lt}(y)\ra \T \,x ({y}/v) )\\
         \tag{$\kfl$6} \T \, x & \Ra \sentkf(x)
    \end{align}
\end{definition}
\noindent In $\kfl$5, $x(y/v)$ denotes the result of substituting, in the formula with code $x$, the variable with code $v$ with the closed term coded by $y$. In particular, in $\kfl$2, $\ulc{\T} \dot{x}\urc$ stands for the result of substituting, in the code of $\T v$, the variable $v$ with the numeral for $x$.

According to Lemma \ref{lem:rechyp} we have that $\bot$, $\supset$ and $\ra$ obey the classical introduction and elimination rules when the antecedent is a formula of $\mc{L}^\ra_{\mathbb{N}}$. An important property of $\kfl$ is that it entails an object-linguistic version of the $\T$-schema for sentences that do not contain the conditional $\ra$. We will return to the philosophical implications of this property in the concluding section. 
\begin{lemma} \label{lemm:disqschem}
    The following are provable in $\kfl$:
    \begin{enumeratei}
        \item ${\tt Sent}_{\lt}(x) \Ra \T \, \ulc{ \neg  \T \, \dot{x}}\urc \lra \T \, \subdot \neg x$
        \item For $A\in \lt$, $\Ra\T\ulc A\urc \lra A$. 
    \end{enumeratei}
\end{lemma}
\begin{proof}
    (i) is immediate by the axioms of $\kfl$, and (ii) is obtained by an external induction on the rank of $A$. 
\end{proof}

\subsection{Semantics} \label{sect:sem}

The intended interpretation of our theory of truth is based on Kripke's fixed point semantics \cite{kri75} and stems from the ${\bf HYPE}$-models presented in Leitgeb \cite[\S7]{lei19}.  Our model will feature a state space, whose states are fixed-points of the usual monotone operator associated with the four-valued evaluation schema as stated in Visser \cite{vis84a} and Woodruff \cite{woo84}. 

Let $\Phi\colon \mc{P}\omega \longrightarrow \mc{P}\omega$ be the operator defined in \cite[Lemma 15.6]{hal14}. States will have the form $(\nat, S)$, where $S$ is a fixed point of $\Phi$. Since we are interested in constant domains and in keeping the interpretation of the arithmetical vocabulary fixed, we omit reference to $\nat$ and identify states with the fixed points themselves. Therefore, we let:
\begin{align}
    &\mathbb{W}:=\{ X \subseteq {\tt Sent}_{\lt}\sth \Phi(X) = X \}, \\
   & S \leq_\mathbb{W} S' :\Leftrightarrow S \subseteq S',    \\
   & S^* := \omega \setminus \overline{S}, \text{ with $\ovl{X}=\{\neg \vphi\sth \vphi\in X\}$}, \\
    & \text{the interpretation of $\T$ is denoted with $\T^S := S$}.
\end{align}
The intended full model $\mathfrak{M}_\Phi$ is then the ${\bf HYPE}$ model based on the frame $( \mathbb{W},\leq_\mathbb{W}, * )$ with the constant domain $\omega$. 
 The intended minimal model $\mathfrak{M}_\Phi^{\tt min}$ is then given by restricting the set of states to the minimal and maximal fixed points. By a straightforward induction on the height of the derivation in $\kfl$, we obtain:
\begin{lemma}
     If $\kfl\vdash \Gamma \Ra \Delta$, then $\mathfrak{M}_\Phi \Vdash \Gamma \Ra \Delta$.
\end{lemma}

\subsection{Proof Theory: Lower Bound}
We show that $\kfl$ can define (and therefore prove the well-foundedness of) Tarskian truth predicates for any $\alpha \prec \varepsilon_0$. By the techniques employed in Feferman and Cantini's analyses of the proof theory of $\kf$ \cite{can89,fef91}, this entails that $\kfl$ can prove ${\tt TI}_{<\vphi_{\varepsilon_0}0}(\lnat)$. 

We first define the Tarskian languages. 
\begin{definition}
    For $0\leq \alpha<\Gamma_0$, we let:
    \begin{align*}
        {\tt Sent}_{\lt}(\ovl{0},x) :\lra\;& {\tt Sent}_{\mc{L}_\nat}(x),\\[1ex]
        {\tt Sent}_{\lt}(\zeta+1,x):\lra \; &{\tt Sent}_{\lt}(\zeta,x)\, \vee \\
                                            & (\exists y \leq x)(x=\corn{{\T} \,\dot{y}} \land  {\tt Sent}_{\lt}(\zeta,y)) \, \vee\\
                                            &(\exists y \leq x)(x=(\subdot{\neg} y) \land  {\tt Sent}_{\lt}(\zeta+1,y))\, \vee\\
                                            &(\exists y, z\leq x)(x=(y\subdot{\vee} z)\land  {\tt Sent}_{\lt}(\zeta+1,y)\land  {\tt Sent}_{\lt}(\zeta+1,z))\,\vee\\
                                            &(\exists v, y\leq x)(x=(\subdot{\forall} v y)\land  {\tt Sent}_{\lt}(\zeta+1,y)),
    \\[1ex]  
    {\tt Sent}_{\lt} (\lambda, x):\lra\;& \exists \, \zeta < \lambda \, {\tt Sent}_{\lt}(\zeta,x) \; \text{ for }\lambda \text{ a limit ordinal.} 
       \end{align*}
        We then write:
    \begin{align*}
    {\tt Sent}_{\lt}^{<\alpha}(x):\lra\;& \exists \zeta \prec \alpha \,{\tt Sent}_{\lt}(\zeta,x), \\
        \T_\alpha (x) :\leftrightarrow \;& {\tt Sent}^{< \alpha}_{\lt} (x) \wedge \T (x).
    \end{align*}

\end{definition}

As we mentioned, the arithmetical vocabulary behaves classically in $\kfl$.
\begin{lemma}\label{lem:prog0}
    $\kfl\vdash \forall x({\tt Sent}_{\mc{L}_\nat}(x)\ra \T \,x\vee \T \, \subdot\neg x)$. 
\end{lemma}

\begin{proof}
    By formal induction on the complexity of the `sentence' $x\in \mc{L}_\nat$. 
\end{proof}

The next two claims establish that the previous fact can be extended to all Tarskian languages whose indices can be proved to be well-founded. First, one shows that the claim `sentences in ${\tt Sent}_{\lt}^{<\eta}$ are either determinately true or determinately false' is progressive.
\begin{lemma}
$\kfl \vdash(\forall \zeta \prec \eta)({\tt Sent}_{\lt}(\zeta,x) \ra \T \,x\vee \T\,\subdot \neg x) \Ra {\tt Sent}_{\lt}({\eta},x) \ra \T \,x\vee \T\, \subdot \neg x$.
\end{lemma}

\begin{proof}
    By the definition of ${\tt OT}$, $\kfl$ proves that $\eta \in {\tt OT}$ is either $0$, or a successor ordinal, or a limit. By arguing informally in $\kfl$, we show that the statement of the lemma holds, thereby establishing the claim. 
    
    Lemma \ref{lem:prog0} gives us the base case. The limit case follows immediately by the definition of ${\tt Sent}_{\lt}(\lambda,x)$. For the successor step, one needs to establish (cf.\cite[Lemma 7]{nic17}):
    \begin{equation}
		 		\label{epsi}{\tt Sent}_{\lt}(\zeta, x)\ra \T x\vee \T \subdot \neg x\Ra {\tt Sent}_{\lt} (\zeta+1, y)\ra  \T y\vee \T \subdot \neg y
	 \end{equation}
	 Claim \eqref{epsi} is obtained by a formal induction on the complexity of $y$. Crucially, the proof rests on the following $\kfl$-derivable claims, which provide the cases required by the induction:
	 \begin{align}
			\label{gamma}&{\tt Sent}_{\lt}{x},\T \,x\vee  \T \,\subdot \neg x\Ra \T \,\subdot \neg x\vee \T \,\subdot \neg \subdot \neg x,\\
			& {\tt Sent}_{\lt}{({x}\subdot \vee{y})}, \T \,x\vee \T\, \subdot \neg x,\T \,y\vee \T \,\subdot \neg y \Ra \T({x}\subdot \vee{y})\vee \T \subdot \neg({x}\subdot \vee{y}),\\
			\label{delta}& {\tt Sent}_{\lt} (\subdot \forall{v}{x}), \forall t \,\T \,x(t/v)\vee \neg \forall t\,\T \, x(t/v)\Ra \T(\subdot \forall {v}{x})\vee  \T( \subdot \neg \forall{v}{x}),\\
			&\T \,x \vee \T \,\subdot \neg x\Ra \T\ulc\T \dot x\urc\vee \T\ulc\neg \T \dot x\urc.
		\end{align}
\end{proof}
By Theorem \ref{thm:genhyp}, we obtain:

\begin{corollary}
For any $\alpha < \varepsilon_0$, $\kfl \vdash \forall x \,({\tt Sent}_{\lt} (\alpha,x)\ra \T \,x \vee \T \,\subdot \neg x)$.
\end{corollary}
Since, by Theorem \ref{thm:genhyp}, $\kfl$ proves transfinite induction up to ordinals smaller than $\varepsilon_0$, it follows that we are able to establish the fundamental properties of Tarskian truth predicates up to any ordinal smaller than $\varepsilon_0$.  For $\alpha<\Gamma_0$, ${\bf RT}_{<\alpha}$ refers to the theory of ramified truth predicates up to $\alpha$, as defined in \cite[\S9.1]{hal14}.
\begin{corollary} \label{cor:lowbound}
    $\kfl$ defines the truth predicates of ${\bf RT}_{<\alpha}$, for $\alpha \prec \varepsilon_0$. 
\end{corollary}
By the proof-theoretic equivalence of systems of ramified truth and ramified analysis established by Feferman \cite{fef64,fef91}, we obtain:
\begin{corollary}
  $\kfl$  proves ${\tt TI}_{<\vphi_{\varepsilon_0}0}(\lnat)$.
 \end{corollary}
Feferman and Cantini established that $\kf$ is proof-theoretically equivalent to ${\bf RT}_{<\varepsilon_0}$. Our results so far then establish that $\kfl$ is proof-theoretically at least as strong as $\kf$. In the next section, we will show that $\kfl$ and $\kf$ are in fact proof-theoretically equivalent.

\subsection{Proof Theory: Upper Bound}

We interpret $\kfl$ in the Kripke-Feferman system $\kf$. For definiteness, we consider the version of $\kf$ formulated in a language $\mc{L}_{{\mbb{T,F}}}$ featuring truth ($\mbb{T}$) and falsity ($\mbb{F}$) predicates. Such a version of $\kf$ is basically the one presented in \cite[\S2]{can89}, but without the consistency axiom that rules out truth-value gluts. 

In order to interpret $\kfl$ into $\kf$, we consider a two-layered translation that differentiates between the external and internal structures of $\lt^\ra$-formulas. Essentially, the external translation fully commutes with negation, and translates the ${\bf HYPE}$ conditional as classical material implication. The internal translation treats negated truth ascriptions as falsity ascriptions, and is defined by an induction on the positive complexity of formulas that adheres to the semantic clauses of ${\bf FDE}$-style fixed-point models. The internal translation therefore translates truth and non-truth of $\kfl$ as truth and falsity of $\kf$, respectively. Since we want to uniformly translate formulas and their codes inside the truth predicate, we  essentially employ the recursion theorem, as described for instance by \cite[\S 5.3]{hal14}.
\begin{definition}\label{def:trlkf}
We define the translations $\tau \colon\lt \longrightarrow \mc{L}_{{\mbb{T,F}}}$, and $\sigma \colon \lt^\ra\longrightarrow \mc{L}_{{\mbb{T,F}}}$ as follows: 
\begin{enumeratei}
    \item 
\begin{align*}
    &( s = t)^\tau= s = t & & ( s \neq t)^\tau= s \neq t \\
    &(\T \,t)^\tau={\mbb T}\tau(t) &&(\neg \T\, t)^\tau= {\mbb F}\tau(t)\\
    &(\neg\neg \vphi)^\tau=(\vphi)^\tau\\
    &(\vphi\vee \psi)^\tau=(\vphi)^\tau\vee (\psi)^\tau&&(\neg(\vphi\vee\psi))^\tau=(\neg \vphi)^\tau\land (\neg \psi)^\tau\\
    &(\forall x\vphi)^\tau=\forall x \vphi^\tau&&(\neg\forall x\vphi)^\tau=\exists x (\neg \vphi)^\tau 
\end{align*}

\item 
\begin{align*}
    &(s = t)^\sigma = s = t \\
    &(\T \,t)^\sigma={\mbb T}\tau(t) && (\neg \T \,t)^\sigma={\mbb F}\tau(t)\\
    &(\neg\vphi)^\sigma=\neg \vphi^\sigma\\
    &(\vphi \vee \psi)^\sigma=(\vphi)^\sigma \vee (\psi)^\sigma&& (\forall x \varphi)^\sigma = \forall x \varphi^\sigma\\
    &(\vphi\ra \psi)^\sigma= \neg (\vphi)^\sigma \vee (\psi)^\sigma
\end{align*}

\end{enumeratei}

\end{definition}

$\kfl$-proofs can then be turned, by the translation $\sigma$, into $\kf$-proofs, as the next proposition shows. 
\begin{proposition}\label{prop:ubkfl}
    If $\kfl\vdash \Gamma \Ra \Delta$, then $\kf \vdash (\bigwedge\Gamma \ra \bigvee \Delta)^\sigma$.
\end{proposition}
\begin{proof}
    The proof is by induction on the height of the derivation in $\kfl$ and follows almost directly from the definition of the translation $\sigma$. Only the case of  ({\bf KFL}3) is slightly more involved: we require that (with $\equiv$ expressing material equivalence):
    \begin{equation}\label{eq:trakfl}
        \kf\vdash {\tt Sent}_{\lt} \Ra {\mbb T}\tau(\subdot \neg x) \equiv {\mbb F}\tau(x).
    \end{equation}
    However, this can be proved by formal induction on the complexity of $x$.
\end{proof}

The combination of Proposition \ref{prop:ubkfl} and Corollary \ref{cor:lowbound} yields that $\kf$ and $\kfl$ have the same arithmetical theorems, and in particular they have the same  proof-theoretic ordinal -- cf.~\cite[\S 6.7]{poh09}.
\begin{corollary}
    $|\kfl|=|\kf|=\vphi_{\varepsilon_0}0$.
\end{corollary}

In the next section we extend our results to schematic extensions of $\kfl$ and $\kf$.

\section{Schematic extension}

\subsection{$\kfls$: Rules and Semantics.} In this section we study the schematic extension of $\kfl$ in the sense of \cite{fef91}. This is obtained by extending $\kfl$ with a special substitution rule that enables us to uniformly replace the distinguished predicate $P$ in arithmetical theorems $A(P)$ of our extended theory for arbitrary formulas of $\lt^\ra$. More precisely, following Feferman, we will employ a schematic language $\lt^\ra (P)$ (and sub-languages thereof) featuring a fresh schematic predicate symbol $P$, which is assumed to behave classically. 
\begin{definition}
    The system $\kfls$ in $\lt^\ra (P)$ extends $\kfl$ with 
    \begin{enumeratei}
        \item $\forall x (P (x) \vee \neg P(x))$;
        \item Disquotational axiom for $P$:
    \begin{align}
 \tag{$\kfl$P}   &\Ra\T (\ulc P  \dot{x} \urc) \lra P(x);
\end{align}
    \item The substitution rule:
\begin{center}
\AXC{$\Ra \forall x ( B(x) \vee \neg B(x))$}
\AXC{$\Gamma(P) \Ra \Delta(P)$}
\RightLabel{for $B$ in $ \lt^\ra (P); \Gamma, \Delta \subseteq \lnat^\ra (P)$.}
\BIC{$\Gamma (B/P) \Ra \Delta (B/P)$}
\DP
\end{center}
    \end{enumeratei}
\end{definition}
 The properties of $\kfl$ expressed by Lemma \ref{lemm:disqschem} transfer directly to $\kfls$, and are proved in an analogous fashion.

The semantics given in \S\ref{sect:sem} can be modified to provide a class of fixed-point models for $\kfls$. We call $\Phi_X$ the result of relativizing the operator from \S\ref{sect:sem} to an arbitrary $X\subseteq \omega$.\footnote{Feferman \cite{fef91} provides a relativized fixed-point construction to arbitrary subsets of natural numbers and establishes the soundness of ${\kf}^*$.} In particular, this means supplementing the positive inductive definition associated with $\Phi$ with the clause:
\begin{quotation}
  a sentence $P z$, with $z$ a closed term of $\lt$, is in the extension of the truth predicate (relativized to $X$) iff ${\tt val}(z)\in X$.
\end{quotation}
This modification clearly does not compromise the monotonicity of the operator. 
Therefore, let ${\tt MIN}_{\Phi_X}$ be the minimal fixed point of $\Phi_X$, and ${\tt MAX}_{\Phi_X}$ its maximal one. For any $X$, we then obtain the minimal ${\bf HYPE}$ model \[
\mf{M}^{\tt min}_{\Phi_X}:=(\{{\tt MIN}_{\Phi_X},{\tt MAX}_{\Phi_X}\},\subseteq, *)
\]
Again in $\mf{M}^{\tt min}_{\Phi_X}$ all arithmetical vocabulary is interpreted standardly at its two states (fixed-points). Only the interpretation of the truth predicate varies. Our notation reflects this. 

\begin{proposition}
    If $\kfls\vdash \Gamma \Ra \Delta$, then for all $X$, $\mf{M}^{\tt min}_{\Phi_X}\Vdash \Gamma \Ra \Delta$.
\end{proposition}
\begin{proof}
 By induction on the length of the derivation in $\kfls$. 
 
 We consider the case of the substitution rule applied to an arithmetical sequent $\Gamma(P) \Ra \Delta(P)$. That is, our proof ends with 
 \[
 \AXC{$\Ra \forall x ( B(x) \vee \neg B(x))$}
\AXC{$\Gamma(P) \Ra \Delta(P)$}
\BIC{$\Gamma (B/P) \Ra \Delta (B/P)$}
\DP
 \]
 with $B(x)$ an arbitrary formula of $\lt^\ra$. 
 
 By induction hypothesis, for all $X$, $\mf{M}^{\tt min}_{\Phi_X}\Vdash  \Gamma(P)\Ra \Delta(P)$. Since $\Gamma(P) \Ra \Delta(P)$ is arithmetical, for all interpretations $Y$ of $P$, $(\mathbb{N}, Y) \vDash \Gamma (P) \Ra \Delta(P)$. Then, following \cite{fef91}, we can let $Y$ to be
 \[
    \{n\sth \mf{M}^{\tt min}_{\Phi_X}\Vdash \Gamma(B({n})/P)\Ra \Delta(B({n})/P)\}
 \]
to obtain that:
\[
    \mf{M}^{\tt min}_{\Phi_X}\Vdash \Gamma(B/P)\Ra \Delta(B/P).
\]
\end{proof}
 
\subsection{Proof-theoretic analysis}
We first consider the proof-theoretic lower-bound for $\kfls$. We adapt to the present setting the strategy outlined in \cite[p.~84]{fest00}. In particular, Feferman and Strahm formalize the notion of $A$-jump hierarchy, which is a hierarchy of sets of natural numbers obtained by iterating an arithmetical operator expressed by an arithmetical formula $A(X,\theta,y)$. An $A$-jump hierarchy is relativized when the starting point is a specific set of natural numbers expressed by some predicate $P$. The notion of $A$-jump hierarchy is quite general, and has as special cases familiar hierarchies such as the Turing-jump hierarchy. 

For our purposes, it's useful to consider $A$-jump hierachies in which membership in second-order parameters is replaced by the notion of satisfaction.  In order to achieve this, we employ Feferman's strategy in \cite{fef91} in which the stages of the Turing jump-hierarchy are represented by means of suitable primitive recursive functions on codes of $\lt$-formulas. Specifically, we encode in suitable primitive recursive functions the stages of a hierarchy in which the formula $A$ is the Veblen-jump formula that will be introduced shortly. 


We denote with ${A}(\T f^A,\eta,y)$ the result of replacing every occurrence of $(u, v) \in X$ in $A(X,\eta,y)$ with 
\[
    \T \;{\rm sub}(f^A_v,\ulc x\urc, {\rm num}(u)),
\]
where the functions $f^A_x(y)$ are recursively defined as follows:
\begin{align*}
f^A_0(x) := & \; \corn{P \dot{ x}},\\
    f^{A,\zeta}(x) := & \; \corn{\dot{x}_0 \prec \dot{\zeta} \wedge \T f^A_{\dot{x}_0} (\dot{x}_1)},\\
    f^A_\zeta(x):=&\; \corn{A(\T f^{A,\dot{\zeta}},\dot{\zeta},\dot x)}.
\end{align*}
In the clause for  $f^{A,\zeta}$, the input $x$ is intended to be an ordered pair $(x_0,x_1)$. 
As in the definitions of translations $\sigma$ and $\tau$ above, the existence of the function $f^A$ can be obtained by employing the recursion theorem, as it needs to apply to its arithmetical code.

Recall the general pattern of the Gentzen jump formula -- with  $\ra$  the ${\bf HYPE}$ conditional:
\[ \mc{J} (B,\xi) := \forall \eta ( \forall \zeta \prec \eta \,  B( \zeta ) \ra \forall \zeta \prec \eta + \xi \, B (\zeta)).  \]
We build our $\mc{A}$-\emph{hierarchy} on the more complex Veblen-jump formula $\mc{A}$, as stated by Sch\"utte in \cite[p.~185]{sch77}, which is crucial for the proof-theoretic analysis of predicative systems.

One first defines the functions:
\begin{align*}
    & {\tt e}(0)=0 && {\tt h}(0)=0\\
    &{\tt e}(\omega^\eta)=\eta && {\tt h}(\omega^\eta)=0\\
    & {\tt e}(\omega^{\eta_1}+\ldots +\omega^{\eta_n})=\eta_n &&{\tt h}(\omega^{\eta_1}+\ldots +\omega^{\eta_n})=\omega^{\eta_1}+\ldots +\omega^{\eta_{n-1}}
\end{align*}
with $\eta_n\preceq \ldots \preceq \eta_1$.

The Veblen-jump formula $\mc{A}$ is then the following:
\[ 
    \mc{A} (\T f^\mc{A} ,\xi, y) := \forall \zeta ({\tt h}(\xi) \preccurlyeq \zeta \prec \xi \, \mc{J} ( \T f^\mc{A}_\zeta , \vphi_{{\tt e}(\xi)}y)).
\]
It expresses that, given some ordinal $\xi$,  the $\mc{A}$-jump hierarchy in the interval between the ordinal ${\tt h}(\xi)$, and $\xi$ itself is closed under the Gentzen jump relative to $\vphi_{{\tt e}(\xi)}y$ (with $y$ a parameter). In the following we will omit the superscripts specifying the formula, since we will keep $\mc A$ fixed.

An essential ingredient of the lower-bound proof for $\kfls$ is the `disquotational' behaviour of our truth predicate for stages in the hierarchy that are provably well-founded. 
\begin{lemma}\label{lem:distru} If we have ${\tt TI}_\alpha(\lt^\ra)$, then for all $\eta$, with $0 \prec \eta \prec \alpha$
\begin{align*}
    \T\subdot{f}_\eta(n) \leftrightarrow \mc{A}(\T\subdot{f}^\eta,\eta,n)
\end{align*}
\end{lemma}
\begin{proof}
 For all $\eta$ and all $n$, we can show that ${\tt Sent}_\eta (f^\eta((n_0,n_1)))$ and ${\tt Sent}_\eta (f_\eta (n))$ by transfinite induction on $\eta$ making use of the properties of the ramified truth predicates such as, for $\lambda\prec \alpha$ limit:
 \[
    \forall \zeta \prec \lambda \big( \T_\lambda (\T_\zeta t)\lra \T_\zeta{\tt val}(t)\big).
 \]
Such properties just state that Tarskian truth predicates are fully compositional for ordinals for which we have transfinite induction \cite{fef91}, \cite[II.9.1]{hal14}. 

Since all truth predicates in $\T\subdot{f}_\eta(n)$ are provably compositional by ${\tt TI}_\alpha(\lt^\ra)$, the claim is obtained by the fact that full compositionality entails disquotation as shown by \cite{tar35}.
\end{proof}

The disquotational properties allow us to establish some fundamental properties of the $\mc{A}$-jump hierarchy. In particular, we show that the $\mc{A}$-jump hierarchy can be elegantly expressed by truth ascriptions on the functions $f_\alpha$.

\begin{lemma}\label{lem:hiert} If ${\tt TI}_\alpha(\lt^\ra)$ is derivable in $\kfls$ for some $\alpha > 0$, then we can derive in $\kfls$:
\begin{align*}
 \forall y ( P (y) \leftrightarrow \T f_0 (y)) 
 \wedge \forall \zeta [& 0 \prec \zeta \prec \alpha \ra 
 \forall y [ \T f_\zeta (y) \leftrightarrow  \forall z ( {\tt h} (\zeta) \preccurlyeq z \prec \zeta \ra \mc{J}(\T f_z,\vphi_{{\tt e}(\zeta)}y)) ]].
 \end{align*}
\end{lemma}
\begin{proof} By our disquotational axioms for $P$, it immediately follows that $\forall y ( P(y) \leftrightarrow  \T f_0 (y)  )$.

Let's assume now that $0 \prec \zeta \prec \alpha$. 
We have 
\[ \T f_\zeta (y) \leftrightarrow \T \corn{ {\mc A} ( \T f^{\dot{\zeta}} , \dot{\zeta} , \dot{y}) }\] 
the right hand side is equivalent by the disquotational property to 
\[ {\mc A} ( \T f^\zeta , \zeta , y). \]
By the definition of $f^\zeta$, the latter formula is in turn equivalent to,
\[ {\mc A} ( \T \corn{\dot{z} \prec \dot{\zeta} \wedge \T f_{\dot{z}} (\dot{x})}, \zeta , y), \]
which is again equivalent to
\begin{equation}\label{eq:dis1}
\forall z ({\tt h}(\zeta) \preccurlyeq z \prec \zeta \,  ( \mc{J}(\T \corn{\dot{z} \prec \dot{\zeta} \wedge \T f_{\dot{z}} (\dot{x})},\vphi_{{\tt e}(\zeta)} y))).
\end{equation}
By applying the disquotational property to \eqref{eq:dis1}, we obtain:
\[ 
    \forall z ({\tt h}(\zeta) \preccurlyeq z \prec \zeta \,  \mc{J}( \T f_{z} (x),\vphi_{{\tt e}(\zeta)} y)). 
\]
\end{proof}

We can now show an analogous claim to Sch\"utte's Lemma 9 in  \cite[p.~186]{sch77}, establishing the progressiveness of the stages of the $\mc{A}$-hierarchy:
\begin{lemma}
If ${\tt TI}_\alpha(\lt^\ra)$ is provable in $\kfls$ for $\alpha<\Gamma_0$, then $\kfls$ proves:
\[
    \forall \zeta (0\prec \zeta\prec \alpha \land (\forall \theta\prec \zeta\, {\tt Prog}(\T f_\theta)\ra{\tt Prog}({\T f_\zeta}))).
\]
\end{lemma}

\begin{proof}
    Let ${\tt l}(\cdot)$ be the function that keeps track of the syntactic complexity of an ordinal code. One first shows that the following claims
    \begin{align}
      & \zeta\prec \alpha \\
      &\forall x \prec \zeta  \,{\tt Prog} (\T f_x)\\
      &\forall y\prec \eta \,\T f_\zeta (y)\\
      & \forall y({\tt l}(y)<{\tt l}(\theta)\ra (y \prec \vphi_{{\tt e}(\zeta)} \eta \ra ( \forall z ( {\tt h} (\zeta ) \preccurlyeq z \prec \zeta \rightarrow \mc{J} ( \T f_z, y)))) \\
      & \theta \prec \vphi_{{\tt e} (\zeta)}\eta\\
      &{\tt h}(\zeta)\preceq \xi\prec \zeta
     \end{align}
     entail $\mc{J}(\T f_\xi,\theta)$. 
     
     First we apply the principle of induction on the syntactic composition of ordinal codes \cite[Thm. 20.10, p.~173]{sch77} (provable in $\kfls$):
     \begin{equation}
        \forall x(\forall y({\tt l}(y)<{\tt l}(x)\ra \phi(y))\ra\phi(x))\ra \phi(t)
     \end{equation}
      to the formula.
     \[
        \phi(u):\lra u\prec \vphi_{{\tt e}(\zeta)}\eta\ra \forall z( {\tt h}(\zeta)\preccurlyeq z \prec \zeta \;\mc{J}(\T f_z,u)),
     \]
     Since we can prove that:
     \[
        \forall x\prec \zeta  \,{\tt Prog}(\T f_x)\ra \big( \forall y\prec \theta \, {\tt l}(y)<{\tt l}(\theta)) \ra (y\prec \vphi_{{\tt e}(\zeta)}\theta \ra \forall z({\tt h}(\zeta)\preccurlyeq z \prec \zeta \;\mc{J}(\T f_z,y)),
     \]
     we obtain by the above-mentioned induction principle:
     \begin{equation}
        \forall x\prec \zeta  \,{\tt Prog}(\T f_x)\ra ( \forall y\prec \eta \,\T f_\zeta(y)\ra ( \theta\prec \vphi_{{\tt e}(\zeta)}\eta\ra \forall z({\tt h}(\zeta)\preccurlyeq z \prec \zeta \;\mc{J}(\T f_z,\theta)) )).
     \end{equation}
     By the definition of progressiveness we obtain:
     \begin{equation}
     {\tt Prog}(\T f_\xi) \ra (\forall x \prec \vphi_{{\tt e}(\zeta)}\eta \,\mc{J} (\T f_\xi, x) \ra \mc{J} ( \T f_\xi,\vphi_{{\tt e}(\zeta)} \eta)).
     \end{equation}
     Therefore, by combining the previous two claims, we obtain:
     \begin{equation}\label{eq:preq1}
         \forall x\prec \zeta \,{\tt Prog}(\T f_x)\ra \big(\forall y\prec \eta \,\T f_\zeta(y)\ra\forall z({\tt h}(\zeta)\preccurlyeq z \prec \zeta \;\mc{J} (\T f_z,\vphi_{{\tt e}(\zeta)}\eta))\big).
     \end{equation}
        Finally, by Lemma \ref{lem:hiert} applied to  \eqref{eq:preq1}, we can conclude that
       \begin{equation}
         \forall x\prec \zeta \,{\tt Prog}(\T f_x)\ra \big(\forall y\prec \eta \,\T f_\zeta(y)\ra \T f_\zeta(\eta))\big).
     \end{equation}

\end{proof}

Let's characterize the fundamental series of ordinals $<\Gamma_0$ as functions of natural numbers in the standard way: $\gamma_0:=\omega$, and $\gamma_{n+1}:= \vphi_{\gamma_n}0$. Then, we have:
\begin{proposition} \label{prop:kflstar}
    If ${\tt TI}_{\gamma_n} (P)$ is derivable in $\kfls$, then ${\tt TI}_{\varphi_{\gamma_n} 0}(P)$ is derivable in $\kfls$.
\end{proposition}
\begin{proof}
 We assume ${\tt Prog}(P)$. If ${\tt TI}_{\gamma_n} (P)$ is derivable in $\kfls$, then by Corollary \ref{cor:clti} and the determinateness of $P$ we can show ${\tt TI}_{\omega^{\gamma_n} +1} (P)$. By the substitution rule we get that the hierarchy predicates are well-defined. Additionally we can prove that 
 \begin{equation} \label{equdefhier}
 \forall \zeta \prec\omega^{\gamma_n} +1 \;\forall x ( \T_\zeta(x) \vee \neg \T_\zeta(x)).
 \end{equation}
   Notice that by \eqref{equdefhier} we can reformulate this fragment of the hierarchy by replacing all the occurrences of the ${ \bf HYPE}$-conditional by the material conditional. Therefore, we have: 
 \begin{equation} \label{equdishier}
     \forall \zeta \prec \omega^{\gamma_n} +1 ( \T f_\zeta (x) \leftrightarrow {\mc A} (\T f^\zeta, \zeta,x) ).
  \end{equation}

 By the previous lemma, for $a\prec \omega^{\gamma_n}+1$, 
\begin{equation}\label{eq:prohier}
    \forall b\prec a \;{\tt Prog}(\T f_b)\ra {\tt Prog}(\T f_a),
\end{equation}
and therefore, by the substitution rule applied to ${\tt TI}_{\gamma_n} (P)$ and \eqref{eq:prohier}, we have that ${\tt Prog}(\T f_{\omega^{\gamma_n}})$, which entails $\T f_{\omega^{\gamma_n}}({0})$.

Using \eqref{equdishier}, we have:
\begin{equation}
    \forall \zeta ({\tt h}(\omega^{\gamma_n})\preccurlyeq \zeta \prec \omega^{\gamma_n} \,\mc{J}(\T f_\zeta,\vphi_{{\tt e}(\omega^{\gamma_n})}0).
\end{equation}
However, since ${\tt h}(\omega^{\gamma_n}) = 0$ and ${\tt e}(\omega^{\gamma_n}) = \gamma_n$ we can then infer 
\begin{equation}
\forall \zeta \prec \omega^{\gamma_n} ( \forall y ( \forall x \prec y  \T f_\zeta (x) \rightarrow \forall x \prec y + \varphi_{\gamma_n}0 \: \T f_\zeta ( x ))).
\end{equation}
By letting $\zeta=y=0$, we obtain $\forall x\prec \varphi_{{\tt e}(\omega^{\gamma_n})}0\;P(x)$, as desired. 
\end{proof}

\begin{corollary}\label{cor:lbkfls}
    $\kfls$ defines the truth predicates of ${\bf RT}_{<\alpha}$, for $\alpha \prec \Gamma_0$. 
\end{corollary}
Following the characterization of predicative analysis in terms of ramified systems given in \cite{fef64,fef91}, and the relationships between ramified truth and ramified analysis studied there, one can then conclude that the systems of ramified analysis below $\Gamma_0$ are proof-theoretically reducible to our system $\kfls$.

The argument employed in the previous section to show that $\kfl$ can be proof-theoretically reduced -- w.r.t.~arithmetical sentences -- to $\kf$ can be lifted to $\kfls$. One can consider the system $\kf^*$ -- ${\tt Ref^*(PA}(P))$ in \cite{fef91} --, and slightly modify the translations $\sigma$, $\tau$ from Definition \ref{def:trlkf}: in particular, we let 
\begin{align*}
    & (P s)^\sigma= (Ps)^\tau= Ps && (\neg Ps)^\tau= \neg Ps. 
\end{align*}
Then, by induction on the length of proof in $\kfls$, we can prove:
\begin{proposition}\label{prop:upschm}
     If $\kfls\vdash \Gamma \Ra \Delta$, then $\kf^* \vdash (\bigwedge\Gamma \ra \bigvee \Delta)^\sigma$.
\end{proposition}

Given the analysis of $\kf^*$ given in \cite{fef91}, the combination of Propositions \ref{prop:upschm} and \ref{prop:kflstar} yields a sharp proof-theoretic analysis also for $\kfls$:
\begin{corollary}
    $|\kfls|=|\kf^*|=\Gamma_0$.
\end{corollary}

\section{Laws of Truth and Intensionality}\label{sec:conc}


The main aim of the paper is to show that $\kfl$ and $\kfls$ are proof-theoretically strong. We now conclude by discussing some of their philosophical virtues. We focus on $\kfl$, but our discussion transfers with little modification to $\kfls$. In particular, we now argue that $\kfl$ displays some advantages with respect to its direct competitors in classical logic ($\kf$) and nonclassical logic ($\pkf$). 

Even truth theorists that consider classical logic as superior do not question the importance of the disquotational intuition for our philosophical notion of truth \cite[p.~189]{fef12}. Theories such as $\kf$ can only approximate such intuition, by restricting it to sentences that are `grounded', in the sense of being provably true or false. $\kfl$ can preserve such intuition in great generality, by validating the $\T$-schema for sentences not containing the conditional $\ra$. Typically, however, nonclassical theories pay tribute to this greater vicinity to the unrestricted $\T$-schema (cf.~Lemma \ref{lemm:disqschem}) with a substantial loss in logical and deductive power. This is not so for $\kfl$: its proof-theoretic strength matches the one of $\kf$.

$\kfl$ appears also to improve on the philosophical rationale behind the fully disquotational truth predicate of $\pkf$. Because of their missing conditional, all variants of $\pkf$ do not have the means to express in the object language their basic principles of truth.  $\kfl$ overcomes these liminations by replacing this metatheoretic inferential apparatus by truth theoretic laws formulated by means of the ${\bf HYPE}$ conditional. This also enables us to formulate fully in the language of $\kfl$ principles of `mixed' nature, such as induction principles open to the truth predicate or, in the case of $\kfls$, additional predicates. This is the root of the increased proof-theoretic strength of $\kfl$. As a consequence, we are also able to speak more fully about the truth (and falsity) of non-semantic sentences of $\lt^\ra$ \cite[p.~391]{lei19}: for instance, if compared to $\pkf$, $\kfl$ can prove many more iterations of the truth predicate over basic non-semantic truths such as $0=0$ (Corollary \ref{cor:lowbound}).

For a full philosophical defence of $\kfl$ -- which, however, is not the main aim of this paper --, it is important to say something about the role of the conditional of ${\bf HYPE}$ and its interaction with the truth predicate of $\kfl$. There are at least two ways of doing so. One could follow Leitgeb in providing a semantic explanation of the intensional nature of the ${\bf HYPE}$ conditional. According to Leitgeb, truth ascriptions are evaluated \emph{locally}, at each fixed-point, whereas conditional statements are evaluated \emph{globally}, that is, by looking at the entire structure of fixed points. Therefore, if the $\T$-schema held also for conditional claims, a truth ascription that contains the conditional would need to be evaluated both locally and globally, which would amount to a category mistake in $\mathfrak{M}_\Phi$. 

Alternatively, one could attempt a direct proof-theoretic explanation of the interaction of the truth predicate and the ${\bf HYPE}$-conditional. Leon Horsten \cite{hor11} defends $\pkf$ on the basis of inferential deflationism: the basic principles of disquotational truth are given inferentially, and essentially so \cite[\S10.2]{hor11}.  Horsten claims in particular that the laws of truth can only be expressed on the background of an inferential apparatus which is not part of the language to which truth is applied. 
One might extend Horsten's inferential approach to the present case, and argue that $\kfl$ characterizes truth in a similar fashion. The laws of truth are given on the background of a theoretical apparatus that essentially involves the conditional of ${\bf HYPE}$. Such theoretical apparatus amounts to the inferential structure of the truth laws of $\pkf$, but now formulated in the language of our theory of truth. This would also explain why the $\T$-schema only holds for sentences that do not contain the ${\bf HYPE}$-conditional. On this picture, one should not expect the conditional of ${\bf HYPE}$ to appear in truth ascriptions in $\kfl$, in the same way as one should not expect inferential devices of $\pkf$ (such as sequent arrows) to appear under the scope of its truth predicate. 

As mentioned, a full philosophical defence of $\kfl$ is outside the scope of our paper, and it will be carried out in future work. 


\bibliographystyle{alpha}
\bibliography{litHB}

\begin{thebibliography}{FHN17}

\bibitem[Can89]{can89}
Andrea Cantini.
\newblock Notes on formal theories of truth.
\newblock {\em Zeitschrift f\"{u}r mathematische Logik und Grundlagen der
  Mathematik}, 35:97--130, 1989.

\bibitem[Fef64]{fef64}
Solomon Feferman.
\newblock Systems of predicative analysis.
\newblock {\em The Journal of Symbolic Logic}, 29(1):1--30, 1964.

\bibitem[Fef84]{fef84}
Solomon Feferman.
\newblock Toward useful type-free theories.i.
\newblock {\em The Journal of Symbolic Logic}, 49(1):75--111, 1984.

\bibitem[Fef91]{fef91}
Solomon Feferman.
\newblock Reflecting on incompleteness.
\newblock {\em The Journal of Symbolic Logic}, 56:1--47, 1991.

\bibitem[FHN17]{fihoni17}
Martin Fischer, Leon Horsten, and Carlo Nicolai.
\newblock Iterated reflection over full disquotational truth.
\newblock {\em Journal of Logic and Computation}, 27(8):2631--2651, 2017.

\bibitem[Fie08]{fie08}
H.~Field.
\newblock {\em Saving Truth from Paradox}.
\newblock Oxford University Press, 2008.

\bibitem[Fie20]{fie20}
H.~Field.
\newblock The power of na\"ive truth.
\newblock {\em The Review of Symbolic Logic}, 2020.
\newblock Forthcoming.

\bibitem[Fis20]{fis20}
Martin Fischer.
\newblock A sequent system for hype.
\newblock draft, 2020.

\bibitem[FS00]{fest00}
Solomon Feferman and T.~Strahm.
\newblock The unfolding of non-finitist arithmetic.
\newblock {\em Annals of Pure and Applied Logic}, 104:75--96, 2000.

\bibitem[G{\"{o}}r71]{goe71}
Sabine G{\"{o}}rnemann.
\newblock A logic stronger than intuitionism.
\newblock {\em The Journal of Symbolic Logic}, 36(2):249--261, 1971.

\bibitem[GSS09]{gaea09}
D.~Gabbay, V.~B. Shehtman, and D.P. Skvortsov.
\newblock {\em Quantification in Nonclassical logic}, volume~1.
\newblock Elsevier Science Publisher, 2009.

\bibitem[Hal14]{hal14}
Volker Halbach.
\newblock {\em Axiomatic Theories of Truth}.
\newblock Cambridge University Press, Cambridge, UK, revised edition, 2014.

\bibitem[HH06]{haho06}
Volker Halbach and Leon Horsten.
\newblock Axiomatizing {K}ripke's theory of truth.
\newblock {\em The Journal of Symbolic Logic}, 71:677--712, 2006.

\bibitem[Hor11]{hor11}
Leon Horsten.
\newblock {\em The Tarskian Turn. Deflationism and Axiomatic Truth}.
\newblock MIT Press, Cambridge, MA, 2011.

\bibitem[HP93]{hapu93}
Petr H\'ajek and Pavel Pudl\'ak.
\newblock {\em Metamathematics of First-Order Arithmetic}.
\newblock Springer Verlag, Berlin, 1993.

\bibitem[Kri75]{kri75}
Saul Kripke.
\newblock Outline of a theory of truth.
\newblock {\em The Journal of Philosophy}, 72:690--716, 1975.

\bibitem[KS94]{kash94}
Ryo Kashima and Tatsuya Shimura.
\newblock Cut-elimination theorem for the logic of constant domains.
\newblock {\em Mathematical Logic Quarterly}, 40:153--172, 1994.

\bibitem[LE83]{lop83}
E.G.K. L\'opez-Escobar.
\newblock A second paper ``on the interpolation theorem for the logic of
  constant domains''.
\newblock {\em The Journal of Symbolic Logic}, pages 595--599, 1983.

\bibitem[Lei19]{lei19}
Hannes Leitgeb.
\newblock Hype: A system of hyperintensional logic (with an application to
  semantic paradoxes).
\newblock {\em Journal of Philosophical Logic}, 48:305--405, 2019.

\bibitem[Nic17]{nic17}
Carlo Nicolai.
\newblock Provably true sentences across axiomatizations of {K}ripke's theory
  of truth.
\newblock {\em Studia Logica}, 106(1):101--130, 2017.

\bibitem[NP01]{nepl01}
Sara Negri and Jan~von Plato.
\newblock {\em Structural Proof Theory}.
\newblock Cambridge University Press, 2001.

\bibitem[Poh09]{poh09}
W.~Pohlers.
\newblock {\em Proof Theory, The first step into impredicativity}.
\newblock Springer Verlag, 2009.

\bibitem[Sch77]{sch77}
Kurt Sch\"{u}tte.
\newblock {\em Proof Theory}.
\newblock Springer Verlag, Berlin, 1977.

\bibitem[Spe20]{spe20}
Stanislav~O. Speranski.
\newblock {\em Negation as a modality in a quantified setting}.
\newblock 2020.
\newblock draft.

\bibitem[Tak87]{tak87}
Gaisi Takeuti.
\newblock {\em Proof Theory}.
\newblock North Holland, Amsterdam, second edition, 1987.

\bibitem[Tar35]{tar35}
Alfred Tarski.
\newblock Der {W}ahrheitsbegriff in den formalisierten {S}prachen.
\newblock In K.~Berka and L.~Kreiser, editors, {\em Logik-Texte}, pages
  445--546. Berlin, 1935.
\newblock 1971.

\bibitem[TS00]{trsc00}
A.S. Troelstra and H.~Schwichtenberg.
\newblock {\em Basic Proof Theory}.
\newblock Cambridge University Press, 2 edition, 2000.

\bibitem[Vis84]{vis84a}
Albert Visser.
\newblock The provability logics of recursively enumerable theories extending
  peano arithmetic at arbitrary theories extending peano arithmetic.
\newblock {\em Journal of Philosophical Logic}, (13):97--113, 1984.

\bibitem[Woo84]{woo84}
Peter~W. Woodruff.
\newblock Paradox, truth and logic part i: Paradox and truth.
\newblock {\em Journal of Philosophical Logic}, 13(2):213--232, 1984.

\end{thebibliography}

\end{document}